\documentclass[11pt,draft,reqno]{amsart}

\usepackage{amsmath,amssymb,amscd,amsfonts,amsthm}
\usepackage{graphics}
\usepackage{dsfont}
\usepackage[all]{xy}
\usepackage{enumerate}

{\catcode`\@=11
\gdef\n@te#1#2{\leavevmode\vadjust{%
 {\setbox\z@\hbox to\z@{\strut#1}%
  \setbox\z@\hbox{\raise\dp\strutbox\box\z@}\ht\z@=\z@\dp\z@=\z@%
  #2\box\z@}}}
\gdef\leftnote#1{\n@te{\hss#1\quad}{}}
\gdef\rightnote#1{\n@te{\quad\kern-\leftskip#1\hss}{\moveright\hsize}}
\gdef\?{\FN@\qumark}
\gdef\qumark{\ifx\next"\DN@"##1"{\leftnote{\rm##1}}\else
 \DN@{\leftnote{\rm??}}\fi{\rm??}\next@}}
\DeclareFontFamily{OT1}{wncyr}{\hyphenchar\font45 }
\DeclareFontShape{OT1}{wncyr}{m}{n}{%
   <5> <6> <7> <8> <9> gen * wncyr
   <10> <10.95> <12> <14.4> <17.28> <20.74>  <24.88>wncyr10}{}
\DeclareFontShape{OT1}{wncyr}{m}{it}{%
   <5> <6> <7> <8> <9> gen * wncyi
   <10> <10.95> <12> <14.4> <17.28> <20.74> <24.88> wncyi10}{}
\DeclareFontShape{OT1}{wncyr}{m}{sc}{%
   <5> <6> <7> <8> <9> <10> <10.95> <12> <14.4>
   <17.28> <20.74> <24.88>wncysc10}{}
\DeclareFontShape{OT1}{wncyr}{b}{n}{%
   <5> <6> <7> <8> <9> gen * wncyb
   <10> <10.95> <12> <14.4> <17.28> <20.74> <24.88>wncyb10}{}
\input cyracc.def
\def\rus{\usefont{OT1}{wncyr}{m}{n}\cyracc\fontsize{9}{11pt}\selectfont}

\DeclareMathSizes{9}{9}{7}{5} 

\theoremstyle{plain}

\newtheorem{theorem}{Theorem}

\newtheorem{remark}{\bf Remark}

\newtheorem*{cornonumber}{Corollary}

\theoremstyle{definition}

\newtheorem{definition}{Definition}

\newtheorem{nothing*}[theorem]{}
\newtheorem{subnothing*}[sub]{}
\newtheorem{example}{Example}

\theoremstyle{remark}

\def\bAn{{\mathbf A}\!^n}
\def\bAm{{\mathbf A}\!^m}
\def\bAr{{\mathbf A}\!^r}
\def\bAt{{\mathbf A}\!^t}

\def\bA1{{\mathbf A}\!^1}
\def\P1{{\bf P}^1}

\def\Autn{{\rm Aut}\,\mathbf A\!^{n}}

\newcommand{\dss}{\hskip -2mm\rotatebox{68}{\raisebox{-1.8\height}{\mbox{\normalsize -\hskip .1mm-\hskip .1mm-}}}\hskip -.6mm}

\begin{document}

\title[Around the Abhyankar--Sathaye Conjecture]{Around the Abhyankar--Sathaye Conjecture}

\author[Vladimir  L. Popov]{Vladimir  L. Popov${}^*$}
\address{Steklov Mathematical Institute,
Russian Academy of Sciences, Gubkina 8, Moscow\\
119991, Russia}
 \email{popovvl@mi.ras.ru}

\address{National Research University\\ Higher School of Economics\\ Myasnitskaya
20\\ Moscow 101000,\;Russia}

\thanks{
 ${}^*$\,Supported by
 grants {\rus RFFI
14-01-00160}, {\rus N{SH}--2998.2014.1}, and the
granting program {\it Contemporary Problems of Theoretical
Mathematics} of the Mathematics Branch of the Russian Academy of Sciences.}


\maketitle

\begin{abstract} A ``rational'' version of the strengthened form of the Commuting Derivation Conjecture,
in which the assumption of
commutativity is dropped, is proved.\;A systematic method of
constructing in any dimension greater than 3
the examples answering in the negative a question by M.\;El\;Kahoui  is developed.
\end{abstract}

\vskip 15mm

\section{Introduction}
Throughout this paper
$k$
stands for an algebraically closed field of characteristic zero which serves as domain of definition for each of the algebraic varieties considered below.

 Recall  that an element $c$ of the polynomial ring $k[x_1,\ldots, x_n]$ in variables
$x_1,\ldots, x_n$ with coefficients in $k$ is called a {\it coordinate} if there are the elements
$t_1,\ldots, t_{n-1}\in k[x_1,\ldots, x_n]$ such that
\begin{equation*}\label{coo}
k[c, t_1,\ldots, t_{n-1}]=k[x_1,\ldots, x_n]
\end{equation*}
(see, e.g.,\,\cite{vdE00}).\;Every coordinate is irreducible and, if  $x_1,\ldots, x_n$ are the standard coordinate functions on the affine space $\bAn$, then the zero locus
$\{c=0\}$
of $c$ in $\bAn$ is isomorphic to  ${\mathbf A}^{\hskip -.3mm n-1}$.\;The
converse is claimed by the classical
\vskip 2mm
\noindent{\bf Abhyankar--Sathaye Conjecture.} {\it If $f\in k[x_1,\ldots, x_n]$ is an irreducible element whose zero locus
in $\bAn$ is isomorphic to ${\mathbf A}^{\hskip -.3mm n-1}$, then $f$ is a coordinate.}

\vskip 2mm

This conjecture is equivalent
to the claim that every closed embedding $\iota\colon {\mathbf A}^{\hskip -.3mm n-1}\hookrightarrow\bAn$ is {\it rectifiable}, i.e., there is an automorphism
$\sigma\in\Autn$ such that
$\sigma\circ\iota\colon {\mathbf A}^{\hskip -.3mm n-1}\hookrightarrow\bAn$ is the standard embedding $(a_1,\ldots, a_{n-1})\mapsto (a_1,\ldots, a_{n-1}, 0)$ (see \cite[Lemma 5.3.13]{vdE00}).

For $n=2$ the Abhyankar--Sathaye conjecture is true
(the Abhyankar--Moh--Suzuki theorem).\;For $n\geqslant 3$ it is still open, though there is a belief that in general it is
false \cite[p.\,103]{vdE00}.

Exploration of this conjecture leads to the problem of constructing closed hypersurfaces in $\bAn$ isomorphic to ${\mathbf A}^{\hskip -.3mm n-1}$, and irreducible polynomials in $k[x_1,\ldots, x_n]$ whose zero loci in $\bAn$ are such hypersurfaces. The following two facts lead, in turn, to the idea of linking this problem with unipotent group actions:
\begin{enumerate}[\hskip 2.2mm\rm(i)]
\item  Every homogeneous space $U/H$, where $U$ is a unipotent algebraic group and $H$ its closed subgroup, is isomorphic to ${\mathbf A}^{\dim U/H}$ (see, e.g.,\,\cite[Prop.\,2(ii)]{Gr58}).
\item  All orbits of every morphic unipotent algebraic group action on a quasi-affine variety $X$ are closed in $X$ (see \cite[Thm.\,2]{Ro61_2}).
\end{enumerate}

In view of  (i) and (ii), every orbit of a morphic unipotent algebraic group action  on $\bAn$ is the image of a closed embedding of
some ${\mathbf A}^{\hskip -.3mm d}$ in $\bAn$.\;In particular, orbits of dimension $n-1$ are the hypersurfaces of the sought-for type.\;Such actions, with a view of getting an approach to the  Abhyankar--Sathaye conjecture, have been the object of study during the last decade, see
\cite{Ma03}, \cite{EK05}, \cite{DEM08}, \cite{DEFM11}.\;In particular, for commutative
unipotent algebraic group actions, the following conjecture   (whose formulation uses the equivalent  language of locally nilpotent derivations, see \cite{F06}) has been put forward:

\vskip 2mm
\noindent {\bf Commuting Derivations Conjecture} (\cite{Ma03}). {\it Let $D$ be a set of $n-1$ commuting locally nilpotent $k$-derivations of $k[x_1,\ldots, x_n]$ linearly independent over $k[x_1,\ldots, x_n]$. Then \begin{equation}\label{c}
\{f\in k[x_1,\ldots, x_n]\mid \partial(f)=0 \;\mbox{for every derivation $\partial\in D$}\}=k[c],
 \end{equation}
 where $c$ is a coordinate in $k[x_1,\ldots, x_n]$.}

\vskip 2mm

This conjecture is open for $n>3$, proved in
 \cite{Ma03} for $n=3$, and follows from Rentschler's theorem \cite{Re68} for $n=2$.\;In \cite[Cor.\,4.1]{EK05} is shown that it is equivalent to a weak version of the Abhyankar--Sathaye conjecture.

On the other hand, in \cite{EK05} is raised the question as to which extent $k[x_1,\ldots, x_n]$ is characterized by property \eqref{c}.\;Namely,
let $\mathcal A$ be a commutati\-ve associative unital
$k$-algebra of transcendence degree $n\!>\!0$ over $k$ such\;that
\begin{enumerate}[\hskip 4.2mm\rm(a)]
\item $\mathcal A$ is a unique factorization domain;
\item
there is a set $D$ of $n-1$  commuting linearly independent over $\mathcal A$ locally nilpotent $k$-derivations of $\mathcal A$.
\end{enumerate}
Consider  the invariant algebra of $D$, i.e., the $k$-algebra
$$
{\mathcal A}^{D}
:=\{a\in \mathcal A\mid \partial(a)=0\;\;\mbox{for every $\partial\in D$}\}.
$$
\vskip 2mm
\noindent
{\bf Question \boldmath
$1$}\,(\cite[p.\,449]{EK05}).
{\it Does the equality}
\begin{equation}\label{cond}
{\mathcal A}^D
=k[c]\;\;\mbox{\it for some element $c\in \mathcal A$}
\end{equation}
{\it imply the existence of elements $s_1,\ldots, s_{n-1}\in \mathcal A$ and $c_1,\ldots, c_{n-1}\in k[c]$ such that
$\mathcal A$ is the polynomial $k$-algebra $k[c, s_1,\ldots, s_{n-1}]$ and $D=\{c_i\partial_{s_i}\}_{i=1}^{n-1}$}?

\vskip 2mm

 This question is inspired by one of the main results of  \cite{EK05},  Theorem 3.1, claiming that for $n=2$ the answer is affirmative.
By
\cite[Thm.\;2.6]{Mi95}, given properties (a) and (b),
equality \eqref{cond} holds and the answer to Question $1$ is affirmative if
 $\mathcal A$ is finitely generated over $k$, the multiplicative group ${\mathcal A}^\star$ of invertible elements of $\mathcal A$ coincides with
$k^\star$, and $n=2$.

The present paper contributes to the Commuting Derivation Conjecture and Question 1.\;In Section \ref{s1} is proved a ``rational'' version of the
 strengthened form of the Commuting Derivation Conjecture, in which the assumption of commutativity is dropped
(see
Theorem \ref{raco}).\;Here ``rational'' means that the notion of ``coordinate'' is replaced by that of
``rational coordinate'' (see Definition  \ref{rcoord}  below).\;Geometrically, the latter means the existence of
a birational (rather than biregular)
automorphism of the ambient affine space that
rectifies the corresponding hypersurface into the standard coordinate hyperplane.\;In Section \ref{s2}, for every  $n\geqslant 4$, is given a systematic method of constructing the pairs $(\mathcal A, D)$, for which the answer to Question 1 is negative.\;Section \ref{rema} contains some remarks.

 \vskip 2mm

\noindent {\it Notation, conventiones, and some generalities}

\vskip 2mm

Below, as in \cite{Bor91}, \cite{Sp98}, ``variety'' means ``algebraic variety'' in the sense of Serre.\;The standard notation and conventions of
\cite{Bor91}, \cite{Sp98},
and \cite{PV94}
are used freely.\;In particular, the algebra of functions regular on a variety $X$ is denoted by $k[X]$ (not by $\mathcal O(X)$ as in
\cite{DEFM11}, \cite{DEM08}).

Given an algebraic variety $Z$, below we denote the Zariski tangent space of $Z$ at a point $z\in Z$ by
${\rm T}_{Z, z}$.

 Let  $G$ be an algebraic group and let $X$ be a variety.
 Given an action
\begin{equation}\label{action}
\alpha\colon G\times X\to X
 \end{equation}
 of $G$ on $X$ and the elements $g\in G$, $x\in X$,  we denote $\alpha(g, x)\in X$ by $g\cdot x$. The $G$-orbit and the $G$-stabilizer of $x$ are denoted resp.\;by $G\cdot x$ and $G_x$.\;If
 \eqref{action} is a morphism, then $\alpha$ is called a {\it regular} (or {\it morphic}) {\it action}.
A regular action $\alpha$ is called {\it locally free} if there is a dense open subset $U$ of $X$ such that the $G$-stabilizer of every point of $U$  is trivial.

Assume that $X$ is irreducible.
The map
\begin{equation}\label{BA}
{\rm Bir}\,X\to {\rm Aut}_k\,k(X), \qquad \varphi\mapsto (\varphi^*)^{-1},
  \end{equation}
  is a group isomorphism.\;We
identify ${\rm Bir}\,X$ and ${\rm Aut}_k\,k(X)$
  by means of
   \eqref{BA} when we
    consider action of a subgroup of ${\rm Bir}\,X$ by $k$-automor\-phisms of $k(X)$ and, conversely,
    action of a subgroup of
  ${\rm Aut}_kk(X)$ by birational automorphisms of\;$X$.

Let $\theta\colon G\to {\rm Bir}\,X$ be an abstract group homomorphism.\;It determines an action of $G$ on $X$ by birational isomorphisms.\;If the domain of definition of the partially defined map $G\times X\to X$, $(g, x)\mapsto \theta(g)(x)$ contains a dense open subset of $G\times X$ and coincides on it with a rational map  $\varrho\colon G\times X\dashrightarrow X$, then
$\varrho$
is called a {\it rational action} of $G$ on $X$.

By \cite[Thm.\;1]{Ros56},
for every rational action
 $\varrho$ there is a regular action of $G$ on an irreducible variety $Y$, the open subsets $X_0$ and $Y_0$ of resp.\;$X$ and $Y$, and an isomorphism $Y_0\to X_0$ such that the induced field isomorphism
$k(X)=k(X_0)\to k(Y_0)=k(Y)$ is $G$-equivariant.

If $\varrho$ is a rational action  of $G$ on $X$,  then by
\begin{equation*}
\pi^{}_{G, X}\colon X\dashrightarrow X\dss G
\end{equation*}
is denoted a rational quotient of $\varrho$, i.e.,  $X\dss G$ and $\pi^{}_{G, X}$ are resp. a variety and a
dominant rational map such that $\pi^{*}_{G, X}(k(X\dss G))=k(X)^G$ (see
\cite[Sect. 2.4]{PV94}).
Depending on the situation we
choose $X\dss G$ as a suitable variety within the class of birationally isomorphic ones.\;A {\it rational section}
for $\varrho$
is a rational map $\sigma\colon X\dss G\dashrightarrow X$ such that $\pi^{}_{G, X}\circ\sigma={\rm id}$.

\section{Rational coordinate}\label{s1}

\begin{definition}\label{rcoord} An irreducible element $c$ of the polynomial ring $k[x_1,\ldots, x_n]$ in variables
$x_1,\ldots, x_n$ with coefficients in $k$ is called a {\it rational coordinate} if there are the elements
$f_1,\ldots, f_{n-1}\in k(x_1,\ldots, x_n)$ such that
\begin{equation}\label{rc}
k(c, f_1,\ldots, f_{n-1})=k(x_1,\ldots, x_n).
\end{equation}
\end{definition}

If $c$ is a rational coordinate and \eqref{rc} holds, then the hypersurface $\{c=0\}$ is birationally isomorphic to ${\mathbf A}^{\hskip -.3mm n-1}$
and the rational map
\begin{equation*}
\tau\colon \bAn \dashrightarrow\bAn,\quad a\mapsto (c(a), f_1(a),\ldots, f_{n-1}(a))
\end{equation*}
 is an element of ${\rm Bir}\,\bAn$.\;Since $k(c, f_1,\ldots, f_{n-1})=k(c, f_1c^{d_1},\ldots, f_{n-1}c^{d_{n-1}})$ for any $d_1,\ldots, d_{n-1}\in \mathbf Z$, we may replace in \eqref{rc} every $f_i$ by an appropriate $f_ic^{d_i}$ and assume that the intersection of $\{c=0\}$ with the domain of definition of $\tau$ is nonempty. Then the image of $\{c=0\}$ under $\tau$ is defined and
its closure is the standard coordinate hyperplane $\{x_1=0\}$. In other words, the hypersurface $\{c=0\}$ is rectified by the birational automorphism $\tau\in {\rm Bir}\,\bAn$.
\begin{theorem}\label{split} Let $\varrho\colon S\times X\dashrightarrow X$ be a rational action of
a connected solvable affine algebraic group $S$
on an irreducible algebraic variety $X$.\;Let
\begin{equation}\label{quo}
\pi^{}_{S, X}\colon X\dashrightarrow X\dss S
 \end{equation}
 be a rational quotient of this action.\;Then there are an integer $m\geqslant 0$ and a birational isomorphism $\varphi\colon X\dss S\times \bAm \dashrightarrow X$ such that the following diagram is commutative
\begin{equation*}
\begin{matrix}
\xymatrix{
X\dss S\times \bAm \ar@{-->}[rr]^{\hskip 6mm\varphi}\ar[rd]_{{\rm pr}_1}&& X\ar@{-->}[dl]^{\pi^{}_{S, X}}\\
&X\dss S&
}
\end{matrix}.
\end{equation*}
\end{theorem}
\begin{proof} Replacing $X$ by a birationally isomorphic variety, we may (and shall) assume
that the action $\varrho$ is regular.
Put
\begin{equation}\label{maxd}
m_{S, X}:=\underset{x\in X}{\max} \dim S\cdot x.
\end{equation}

First, consider the case
\begin{equation}\label{1dim}
\dim S=1.
\end{equation}
In this case
$m_{S, X}\leqslant 1$.\;If $m_{S, X}=0$, the action $\varrho$ is trivial, hence $X\dss S=X$, $\pi^{}_{S, X}={\rm id}$, and the claim is clear.\;Now let $m_{S, X}=1$.\;This means that $S$-stabilizers of points of a dense open subset are finite.\;In this case, we may assume that
\begin{equation}\label{rho}
\mbox{the action
$\varrho\,$ is locally free.}
\end{equation}

To prove this claim, recall
(see, e.g.,\,\cite[Thm.\,3.4.9]{Sp98}) that, given
\eqref{1dim}, we have $S={\mathbf G}_{a}$ or ${\mathbf G}_{m}$.\;If $S={\mathbf G}_{a}$,
then the claim follows from the fact that, due to the assumption ${\rm char}\,k=0$, there are no nontrivial finite subgroups in $S$.\;If $S={\mathbf G}_{m}$, then $S/F$ is isomorphic to $S$ for any finite subgroup $F$, see, e.g., \cite[2.4.8(ii) and 6.3.6]{Sp98}.\;Therefore, taking as $F$ the kernel of $\varrho$, we may assume that
 $\varrho$ is faithful.\;Since
 $S$ is a torus, this, in turn, implies that $\varrho$ is locally free, see \cite[Lemma 2.4]{P13}.\;Thus \eqref{rho} holds.

 Given \eqref{rho}, by \cite[Thm.\,2.13]{CTKPR11}  we may replace $X$ by an appropriate $S$-invariant open subset and assume  that \eqref{quo} is a torsor.\;Since $S$ is a connected solvable affine algebraic group, by \cite[Thm.\,10]{Ros56} this torsor admits a rational section and therefore is trivial over an open subset of $X\dss S$.
  As the group variety of $S$ is birationally isomorphic to
 $\bA1$, this completes the proof of theorem in the case when \eqref{1dim} holds.

In the general case we argue by induction on $\dim S$.\;If $\dim S>0$, then solvability of $S$ yields the
existence of a closed connected normal subgroup $N$ in $S$ such that the (connected solvable affine) algebraic group $G:=S/N$ is one-dimensional.\;Put $Y:=X\dss N$.\;By the inductive assumption, there are an integer $r\geqslant 0$ and a birational isomorphism $\lambda\colon Y\times \bAr \dashrightarrow X$ such that the following diagram is commutative
\begin{equation}\label{N}
\begin{matrix}
\xymatrix{
Y\times \bAr \ar@{-->}[rr]^{\hskip 6mm\lambda}\ar[rd]_{{\rm pr}_1}&& X\ar@{-->}[dl]^{\pi^{}_{N, X}}\\
&Y&
}
\end{matrix}.
\end{equation}

 Since $N\lhd S$ and $\pi^{*}_{N, X}(k(Y))=k(X)^N$, the action $\varrho$ induces a rational action of $G$ on $Y$ such that
 \begin{align}
 Y\dss G&=X\dss S,\label{ide}\\
 \pi^{}_{S, X}&=\pi^{}_{G, Y}\circ\pi^{}_{N,X}.\label{ide2}
 \end{align}

 Given \eqref{ide} and using the proved validity of theorem for one-dimensional groups, we obtain that there are an integer $t\geqslant 0$ and a birational isomorphism
 $\gamma\colon X\dss S\times \bAt \dashrightarrow Y$ such that the following diagram is commutative
\begin{equation}\label{gG}
\begin{matrix}
\xymatrix{
X\dss S\times \bAt \ar@{-->}[rr]^{\hskip 6mm\gamma}\ar[rd]_{{\rm pr}_1}&& Y\ar@{-->}[dl]^{\pi^{}_{G, Y}}\\
&X\dss S&
}
\end{matrix}.
\end{equation}

From \eqref{ide2} and diagrams \eqref{N},
\eqref{gG} we see that one can take $m=r+t$ and $\varphi=\lambda\circ (\gamma\times {\rm id}_{\bAr})$.\;This completes the proof.
 \quad $\square$ \renewcommand{\qed}{}\end{proof}

 \begin{remark}\label{dimens}
 {\rm The number $m$ in the formulation of Theorem \ref{split} is equal to the number $m^{}_{S,X}$ given by
 \eqref{maxd}.}
 \end{remark}
 \begin{cornonumber}\label{cor} In the notation of Theorem {\rm \ref{split}}, there are the elements
 $f_1,\ldots, f_m$ of $k(X)$  such that
 \begin{enumerate}
 \item[\rm(i)] $f_1,\ldots, f_m$ are algebraically independent over $k(X)^S$;
 \item[\rm(ii)]  $k(X)=k(X)^S(f_1,\ldots, f_m)$.
 \end{enumerate}
 \end{cornonumber}

\begin{theorem}\label{raco}
Let a unipotent affine algebraic group $U$ regularly act on $\bAn$.\;If
\begin{equation}\label{md}
\underset{a\in \bAn}{\max} \dim U\cdot a=n-1,
 \end{equation}
 then $k[\bAn]^U=k[c]$, where $c$ is a rational coordinate in $k[\bAn]$.
\end{theorem}
\begin{proof} By Rosenlicht's theorem \cite[Thm.\,2]{Ros56} and the fiber dimension theorem, \eqref{md} implies that the transcendence degree of $k(\bAn)^U$ over $k$ is $1$ (cf.\;\cite[Sect.\,2.3, Cor.]{PV94}).\;Since $U$ is unipotent,
$k(\bAn)^U$ is the field of fractions of $k[\bAn]^U$,  see
\cite[p.\,220, Lemma]{Ro61_2}.\;By \cite{Za54} these properties imply that
$k[\bAn]^U$ is a finitely generated $k$-algebra.\;Integral closedness of
$k[\bAn]$ yields integral closedness of $k[\bAn]^U$, see \cite[Thm.\;3.16]{PV94}.\;Thus $\bAn/\!\!/U:={\rm Spec}\,
k[\bAn]^U$ is an irreducible smooth affine algebraic curve.\;This curve is rational by L\"uroth's theorem because  $k(\bAn)^U$ is a subfield of $k(\bAn)$.\;We then conclude that $\bAn/\!\!/U$ is obtained  from $\P1$ by removing $s\geqslant 1$ points.\;Since $k[\bAn/\!\!/U]^\star=k^\star$, we have  $s=1$, i.e., $\bAn/\!\!/U=\bA1$, or, equivalently, $k[\bAn]^U=k[c]$ for an element $c\in k[\bAn]$.\;Since the group $U$ is unipotent, it is connected (in view of ${\rm char}\,k=0$) and
 admits no nontrivial algebraic homomorphisms  $U\to {\mathbf G}_m$.\;By \cite[Thm.\;3.1]{PV94}, this implies that every nonconstant irreducible element of $k[\bAn]$ dividing $c$ lies in $k[\bAn]^U$, which, in turn, easily implies irreducibility of $c$.

We now claim that $c$ is a rational coordinate in $k[\bAn]$.\;Indeed, since $k(\bAn)^U$ is the field of fractions of $k[\bAn]^U$, we have
$k(\bAn)^U=k(c)$.\;Hence by \eqref{md}, Remark \ref{dimens}, and Corollary of Theorem \ref{split}, there are the elements
$f_1,\ldots, f_{n-1}\in k(\bAn)$ such that $k(\bAn)=k(c,  f_1,\ldots, f_{n-1})$.\;Whence the claim by
Definition \ref{rcoord}.
\quad $\square$ \renewcommand{\qed}{}\end{proof}

\section{Commuting derivations of unique factorization domains}\label{s2}

First, we shall introduce the notation.\;Let $G$ be a connected simply connected semisimple algebraic group.
Fix a maximal torus $T$ of $G$. Let ${\rm X}$ and ${\rm X}^\vee$ be respectively the character lattice and the cocharacter lattice of $T$ in additive notation, and let $\langle\ {,}\ \rangle\colon {\rm X}\times {\rm X}^\vee\to \mathbf Z$ be
the natural pairing.\;The value of an element $\varphi\in {\rm X}$ at a point $t\in T$ denote by
 $t^\varphi$.\;Let
$\Phi$ and $\Phi^+\subset {\rm X}$ respectively be the root system of $G$ with respect to $T$ and the system of positive roots
of
$\Phi$ determined by a fixed Borel subgroup $B$ of $G$ containing $T$.\;Given a root $\alpha\in \Phi$, denote by ${\alpha}^\vee\colon {\bf G}_m\to T$ and $U_\alpha$ respectively
  the coroot and the one-dimensional unipotent root subgroup of $G$
corresponding to $\alpha$.

Let
$\Delta=\{\alpha_1,\ldots, \alpha_r\}$ be the system of simple roots of $\Phi_+$
indexed as in \cite{B68}.\;If $I$ is a subset
of $\Delta$, let $\Phi_I$  be the set of elements of $\Phi$  that are linear
combinations of the roots in $I$.\:Denote by  $L_I$ be the subgroup of $G$ generated by $T$
and all the $U_\alpha$'s with $\alpha\in \Phi_I$.\;Let $U_I$ (respectively, $U^-_I$)
be the subgroup of
$G$ generated by all the $U_\alpha$'s with $\alpha\in \Phi^+\setminus \Phi_I$
(respectively, $-\alpha\in \Phi^+\setminus \Phi_I$ ). Then
$P_I := L_IU_I$ and $P_I^-:=L_IU_I^-$
are parabolic subgroups of $G$ opposite to one
another, $U_I$ and $U_I^-$
are the unipotent radicals of $P_I$ and $P_I^-$ respectively,
$L_I$ is a Levi subgroup of $P_I$ and $P_I^-$, and
\begin{align}\label{U}
\dim U_I&=\dim U^-_I=|\Phi^+\setminus\Phi_I|,\\
\dim G &= \dim L_I + 2 \dim U_I^-.\label{G}
\end{align}
Every closed subgroup of $G$ containing $B$ is of the form $P_I$ for some $I$. Every parabolic subgroup of $G$ is conjugate to a unique $P_I$,
called {\it standard} (with respect to $T$ and $B$); see, e.g.,\;\cite[8.4.3]{Sp98}.

Let
$\mathcal D\subset {\rm X}$ be the monoid of highest
weights (with respect to $T$ and $B$) of simple $G$-modules.\;Given a weight $\varpi\in \mathcal D$,
let $E(\varpi)$ be a simple $G$-module
with $\varpi$ as the highest weight.

Denote by $\varpi_1,\ldots, \varpi_r$
the system of
all indecomposable elements
  (i.e.,
  {\it fundamental weights})  of $\mathcal D$ indexed
in such a way that
\begin{equation}\label{delta}
\langle \varpi_i, {\alpha}^\vee_j\rangle=\delta_{ij}.
\end{equation}
This system
freely
generates $\mathcal D$, i.e., for every weight $\varpi\in \mathcal D$ there are uniquely defined
nonnegative integers $m_1,\ldots, m_r$ such that $\varpi=m_1\varpi_1+\cdots+m_r\varpi_r$.\;By virtue of \eqref{delta},
\begin{equation}\label{numlab}
\langle\varpi, \alpha^\vee_i\rangle=m_i.
\end{equation}
The integers \eqref{numlab}
are called the {\it numerical labels of} $\varpi$.
 The ``labeled'' Dynkin diagram of
$\alpha_1,\ldots, \alpha_r$, in which $m_i$ is the label  of the
node $\alpha_i$ for every $i$,
is called {\it the Dynkin diagram of} $\varpi$.

  Given a nonzero $\varpi\in \mathcal D$, denote by  ${\mathbf P}(E(\varpi))$ the projective space of all one-dimensional linear subspaces of $E(\varpi)$.\;The
natural projection
$$
\pi\colon E(\varpi)\setminus \{0\}\to {\mathbf P}(E(\varpi))
$$
is $G$-equivariant with respect to the
natural action of $G$ on
${\mathbf P}(E(\varpi))$.\;The
fixed point set of $B$ in
${\mathbf P}(E(\varpi))$ is
a single point
$p(\varpi)$ and the
$G$-orbit $\mathcal O(\varpi)$ of $p(\varpi)$
is the unique
closed $G$-orbit in ${\mathbf P}(E(\varpi))$.

Consider in  $E(\varpi)$
the affine cone $X(\varpi)$ over $\mathcal O(\varpi)$, i.e.,
\begin{equation}\label{X}
X(\varpi)=\{0\}\sqcup\pi^{-1}(\mathcal O(\varpi)),
\end{equation}
It is a $G$-stable irreducible
closed subset of $E(\varpi)$.\;Let
 ${\mathcal A}(\varpi)$ be the coordinate algebra of
$X(\varpi)$:
$$
{\mathcal A}(\varpi)=k[X(\varpi)],
$$
 and let $n$ be the transcendence degree of ${\mathcal A}(\varpi)$ over $k$,
i.e.,
\begin{equation}\label{n}
n=\dim X(\varpi).
\end{equation}

Since every $U_\alpha$
is a one-dimensional unipotent group,
its natural action on
$X(\varpi)$
determines an algebraic vector field ${\mathcal F}_\alpha$ on $X(\varpi)$, which, in turn,
determines
a locally nilpotent derivation $\partial_\alpha$ of ${\mathcal A}(\varpi)$;
see \cite[1.5]{F06}.\;Actually,  $\partial_\alpha$ is induced by a locally nilpotent derivation of $k[E(\varpi)]$. Namely, as above, the natural action of $U_{\alpha}$ on $E(\varpi)$ determines a locally nilpotent derivation $D_\alpha$ of $k[E(\varpi)]$.\;Since the ideal ${\mathcal I}(\varpi)$ of $X(\varpi)$ in $k[E(\varpi)]$ is $D_\alpha$-stable, $D_\alpha$ induces a locally nilpotent derivation of
$\mathcal A(\varpi)=k[E(\varpi)]/{\mathcal I}(\varpi)$; the latter is $\partial_\alpha$.

\eject

\begin{theorem} \label{method}  For every nonzero weight $\varpi\in\mathcal D$, the following hold{\rm:}

\begin{enumerate}[\hskip 2.5mm \rm(i)]
\item The stabilizer $G_{p(\varpi)}$ of $p(\varpi)$ in $G$ is $P_{I(\varpi)}$, where
\begin{equation}\label{I}
I(\varpi)=\{\alpha\in\Delta\mid \langle\varpi, \alpha^\vee\rangle=0\}.
\end{equation}
\item $\dim U^-_{I(\varpi)}=\Phi^+\setminus\Phi_{I(\varpi)}=n-1$.
 \item
The stabilizer of a point in general position for the natural action of $U^-_{I(\varpi)}$ on $X(\varpi)$
is trivial.
\item
The set $\{\partial_{-\alpha} \mid \alpha \in \Phi^+\setminus\Phi_{I(\varpi)}\}$ of $n-1$ locally nilpotent derivations
 of the algebra ${\mathcal A}(\varpi)$ is linearly independent over ${\mathcal  A}(\varpi)$.
 \item  The following properties are equivalent{\rm:}
 \begin{enumerate}[\hskip .9mm \rm(a)]
 \item[\rm(C)]
 $\{\partial_{-\alpha} \mid \alpha \in \Phi^+\setminus\Phi_{I(\varpi)}\}$  is the set of commuting derivations.\;Equi\-va\-lently,  the unipotent group $U^-_{I(\varpi)}$ is commutative.

\item[\rm(D)]
In the Dynkin diagram of $\varpi$, every connected component $S$
has at most one
node with a nonzero label, and if such a
node $v$ exists, then
$S$ is not of type ${\sf E}_8$, ${\sf F}_4$, or ${\sf G}_2$,
and
$v$ is a black
node of $S$ colored as in the following\;table{\rm:}

\begin{center}
\vskip 3mm

\hskip -15mm\begin{tabular}{c|c}
\text{ {\small\rm type of} $S$} & {\small\rm colored} $S$
\\[2pt]
 \hline
 &\\[-9pt]
 ${\sf A}_l$ &$\begin{matrix}\xymatrix@=4.5mm@M= -.4mm{\bullet\ar@{-}[r] &\bullet\ar@{-}[r] &{\ \cdots\
} \ar@{-}[r]&\bullet\ar@{-}[r] &\bullet}\end{matrix}$\\
  ${\sf B}_l$ &$\begin{matrix}\xymatrix@=4.5mm@M=
-.4mm{\bullet \ar@{-}[r]&\circ \ar@{-}[r]& {\
\cdots \ }\ar@{-}[r]&\circ\ar@{=>}[r]&\circ}\end{matrix}$\\
   ${\sf C}_l$ &$\begin{matrix}\xymatrix@=4.5mm@M= -.4mm{\circ\ar@{-}[r]
   &\circ\ar@{-}[r]
   &{\ \cdots\ } \ar@{-}[r]&\circ\ar@{<=}[r]&
\bullet}\end{matrix}$\\
    ${\sf D}_l$ &$\begin{matrix}
\xymatrix@=4.5mm@M= -.4mm@R=3mm{ &&&&\bullet
\\
\bullet\ar@{-}[r]& \circ\ar@{-}[r]&{\ \cdots \
}\ar@{-}[r] & \circ\ar@{-}[ur] \ar@{-}[dr]
\\
&&&&\bullet }\end{matrix}$
\\
     ${\sf E}_6$ &\hskip -5mm$\begin{matrix}\xymatrix@=4.5mm@M= -.4mm{\bullet\ar@{-}[r] &
     \circ\ar@{-}[r] &
     \circ\ar@{-}[r] \ar@{-}[d]&
     \circ\ar@{-}[r] &
     \bullet\\
     &&\circ&&}\end{matrix}$\\
      ${\sf E}_7$ &
      $\begin{matrix}\xymatrix@=4.5mm@M= -.4mm{\circ\ar@{-}[r] &
     \circ\ar@{-}[r] &
     \circ\ar@{-}[r] \ar@{-}[d]&
     \circ\ar@{-}[r] &\circ\ar@{-}[r] &
     \bullet\\
     &&\circ&&}\end{matrix}$
\end{tabular}
\end{center}
\vskip 3mm
\end{enumerate}
\item ${\mathcal  A}(\varpi)^\star=k^\star$.

\item ${\mathcal  A}(\varpi)$ is a unique factorization domain if and only if $\varpi$ is a fundamental weight.

\item The following properties are equivalent{\rm:}
\begin{enumerate}
\item[$({\rm s}_1)$] $X(\varpi)$ is singular{\rm;}
\item[$({\rm s}_2)$] $\dim E(\varpi)\!>\!n${\rm;}
\item[$({\rm s}_3)$] $X(\varpi)\neq E(\varpi)$.
\end{enumerate}
The singular
locus of every singular $X(\varpi)$ is the vertex $0$.
\end{enumerate}
\end{theorem}

\begin{proof}
(i): By the definition of $p(\varpi)$, the group $B$ is contained in $G_{p(\varpi)}$. Hence
\begin{equation}\label{p}
G_{p(\varpi)}=P_I\quad \mbox{for some $I$.}
\end{equation}

In order to prove (i),
fix a point $v\in \pi^{-1}(p(\varpi))$ and denote by $G_v$ its stabilizer in $G$ and by $\ell$ the line
$\pi^{-1}(p(\varpi))\cup\{0\}$ in $E(\varpi)$.\;We first show that the following properties of
a root $\alpha\in \Delta$ are equivalent:
\begin{enumerate}[\hskip 4.2mm\rm(a)]
\item $\alpha\in I$;
\item $\langle \varpi, \alpha^\vee\rangle = 0$;
\item the image of $\alpha^{\vee}$ is contained in $G_v$.
\end{enumerate}

The definitions of $p(\varpi)$ and $v$ imply that
\begin{equation}\label{t}
t\cdot v=t^\varpi v \quad \mbox{for every element $t\in T$},
\end{equation}
and the definition of $\langle\ {,}\ \rangle$ entails the equality
\begin{equation}\label{s}
(\alpha^\vee(s))^\varpi=s^{\langle\varpi,\alpha^\vee\rangle}\quad\mbox{for every element $s\in{\mathbf G}_m$}.
\end{equation}
Combining \eqref{t} and \eqref{s}, we obtain the equivalence (b)$\Leftrightarrow$(c).

(a)$\Rightarrow$(c): By \eqref{p}, the line $\ell$ is stable with respect to $U_\alpha$.\;Being unipotent, the group $U_\alpha$ has no nontrivial characters and, therefore, no nontrivial one-dimensional modules.\;This proves that  $U_\alpha$ is contained in $G_v$.

If (a) holds, then by \eqref{p} the line $\ell$ is stable with respect
to $U_{-\alpha}$ as well.\;The same argument as for $U_\alpha$ then shows that
$U_{-\alpha}$ is contained in $G_v$.\;Hence
$G_v$ contains the group $S_\alpha$ generated by $U_\alpha$ and $U_{-\alpha}$.\;But $S_\alpha$ contains
the image of $\alpha^\vee$.\;This proves the implication (a)$\Rightarrow$(c).

(c)$\Rightarrow$(a): Assume that (c) holds.\;Since, as explained above, $U_\alpha$ is contained in $G_v$,
the subgroup of $S_\alpha$ generated by $U_\alpha$ and the image of $\alpha^\vee$ is contained in $G_v$.\;This subgroup is a Borel subgroup of $S_\alpha$.\;Therefore the $S_\alpha$-orbit of $v$ is a complete subvariety of $E(\varpi)$, i.e., a point.\;This means that $S_\alpha$ is contained in $G_v$.\;Therefore,
$U_{-\alpha}$ is contained in $G_v$; whence (a) holds.\;This proves the implication
(c)$\Rightarrow$(a).

\smallskip

Combining  now \eqref{p} and \eqref{I} with the equivalence (a)$\Leftrightarrow$(c), we obtain the proof of part (i).

\smallskip

(ii): Since $X(\varpi)$ is the affine cone over $\mathcal O(\varpi)$, we have
\begin{equation}\label{cone}
\dim X(\varpi)=\dim \mathcal O(\varpi)+1.
\end{equation}
On the other hand,
\eqref{U}, \eqref{G}, and (i) entails
\begin{equation}\label{dO}
\dim \mathcal O(\varpi)=\dim U^-_{I_\varpi}.
\end{equation}
Combining \eqref{cone}, \eqref{dO}, and \eqref{n}, we obtain the proof of part (ii).

\smallskip

(iii): Since $U^-_I\cap P_I=\{e\}$
for every $I$, the stabilizer of $v$ for the natural action of  $U^-_{I_\varpi}$ on $X(\varpi)$
is trivial because of (i).\;Hence $\dim X(\varpi)$ is the maximum of dimensions of $U^-_{I_\varpi}$-orbits in
$X(\varpi)$.\;Since $U^-_{I_\varpi}$-orbits of points of
a dense open subset of  $X(\varpi)$ have maximal dimension, this means that
the $U^-_{I_\varpi}$-stabilizer of a point
in general position in  $X(\varpi)$ is finite.\;But $U^-_{I_\varpi}$ has no nontrivial
finite subgroups because it
is a connected unipotent group
and ${\rm char}\,k=0$.\;This proves
part (iii).

\smallskip

(iv): Given a point $a\in X(\varpi)$, denote  its
 $U^-_{I(\varpi)}$-orbit by $U^-_{I(\varpi)}\cdot a$.\;By (iii), taking a suitable $a$, we may assume that
 \begin{equation}\label{tang}
 \dim {\rm T}_{U^-_{I(\varpi)}\cdot a, a}=\dim U^-_{I(\varpi)}.
 \end{equation}
 Since $U^-_{I(\varpi)}=\prod_{\alpha\in \Phi^+\setminus \Phi_{I(\varpi)}} U_{-\alpha}$ (the product being taken in any order), and ${\rm char}\,k=0$,
\begin{equation}\label{F}
{\rm T}_{U^-_{I(\varpi)}\cdot a, a}=\mbox{the linear span of $\{\mathcal F_{-\alpha}(a)\mid \alpha\in  \Phi^+\setminus \Phi_{I(\varpi)}\}$ over $k$}.
\end{equation}
It follows from \eqref{tang}, \eqref{F}, and (ii) that all the vectors
$\mathcal F_{-\alpha}(a)$, where $\alpha\in  \Phi^+\setminus \Phi_{I(\varpi)}$ are linearly independent over $k$.\;Hence all the vector fields $\mathcal F_{-\alpha}$, where $\alpha\in  \Phi^+\setminus \Phi_{I(\varpi)}$,
are linearly independent over $A(\varpi)$.\;This proves part (iv).
\smallskip

(v):
Since standard parabolic subgroups of $G$ are products of standard parabolic subgroups of connected simple normal subgroups of $G$, the proof is reduced to the case, where $G$ is simple.\;In this case
 (C)$\Leftrightarrow$(D) follows from \eqref{I} and the known
 classification of parabolic subgroups that have commutative unipotent radical (see, e.g., \cite[Lemma 2.2 and Rem.\;2.3]{RRS92}).

\smallskip

(vi):
Since
$\varpi\neq 0$,\;the action of $T$ on
$\pi^{-1}(p(\varpi))$ is nontrivial and, therefore, transitive.\;Since the restriction of $\pi$ to $X(\varpi)\setminus \{0\}$ is a $G$-equivariant morphism
onto the orbit $O(\varpi)$, this entails that
 \begin{equation}\label{orb}
 G\cdot v=X(\varpi)\setminus \{0\}.
 \end{equation}

By \cite[Thm.\;2]{PV72},
\begin{equation*}
A(\varpi)\to  k[G\cdot v],\;\;f\mapsto f|_{G\cdot v}
\end{equation*}
is an isomorphism of $k$-algebras.\;On the other hand, the orbit map $G\to G\cdot v$ induces
the embedding of $k[G\cdot v]$ into $k[G]$, the coordinate algebra of $G$.\;By \cite[Thm.\;3]{Ro61_1}, every element of
$k[G]^\star$ is of the form $cf$, where $c\in k^\star$ and $f\colon G\to k^\star$ is a character of $G$.\;Being connected semisimple, $G$ has no nontrivial characters; whence $k[G]^\star=k^\star$.\;This proves part (vi).

\smallskip

(vii):  This is proved, basing on \cite{P74}, in \cite[Thms.\;4 and 5]{PV72}.

\smallskip

(viii):
By virtue of  \eqref{orb}, the singular locus of
$X(\varpi)$ is either $\{0\}$ or empty.\;In par\-ticular, $X(\varpi)$ is singular if and only if
$0$ is the singular point of $X(\varpi)$, i.e., if and only if $\dim {\rm T}_{X(\varpi), 0}>n$.\;The inclusion $X(\varpi)\subseteq E(\varpi)$ yields
the inclusion ${\rm T}_{X(\varpi), 0}\subseteq {\rm T}_{E(\varpi), 0}=E(\varpi)$,
and
since $0$ is a $G$-fixed point,
${\rm T}_{X(\varpi), 0}$ is a submodule of the $G$-module $E(\varpi)$.\;As
the latter is simple,  ${\rm T}_{X(\varpi), 0}\!=\!E(\varpi)$.\;This proves $({\rm s}_1)$$\Leftrightarrow$$({\rm s}_3)$.\;As $({\rm s}_2)$$\Leftrightarrow$$({\rm s}_3)$ is clear, this
completes the proof of (viii).
\quad $\square$ \renewcommand{\qed}{}\end{proof}

Thus, for every fundamental weight $\varpi$ such that
\begin{enumerate}[\hskip 4.2mm ---]
\item the property specified in Theorem \ref{method}(v)(D) holds;
\item the variety $X(\varpi)$ is singular,
\end{enumerate}
the answer to Question 1 for the pair $(\mathcal A, D)$, where
\begin{align*}
\mathcal A&:=\mathcal A(\varpi), \\[-1mm]
D&:=\{\partial_{-\alpha} \mid \alpha \in \Phi^+\setminus\Phi_{I(\varpi)}\},
\end{align*}
is negative.\;There are examples of such pairs in any dimension $n\geqslant 4$.\;

\begin{example}\label{ex1}
Let $G$ be of type ${\sf D}_\ell$, $\ell\geqslant 3$, and $\varpi=\varpi_1$.\;Denote by $V$ be the underlying vector of $E(\varpi)$ and by $\varphi_\varpi\colon G\to {\rm GL}(V)$ the homomorphism determining the $G$-module structure of $E(\varpi)$.\;Then  $\dim V=2\ell$
and  $\varphi_\varpi(G)$ is the orthogonal group
of a nondegenerate quadratic form $f$ on $V$. There is a basis
\begin{equation}\label{basis}
e_1,\ e_2,\ldots, e_\ell, e_{-\ell}, e_{-\ell+1},\ldots, e_{-1}
\end{equation}
of $V$ such that
$$
f=x_{-1}x_1+x_{-2}x_2+\cdots + x_{-\ell}x_\ell,
$$
where $x_i$  is the $i$th coordinate function on $V$ in basis \eqref{basis}.\;The variety $X(\varpi)$ coincides with that of all isotropic vectors of $f$,
$$
X(\varpi)=\{v\in V\mid f(v)=0\},
$$
which, in turn, coincides with the closure of the $G$-orbit of $e_1$.\;Hence, if
$\,{\mathcal P}_{2\ell}$ is the polynomial ring in $2\ell$ variables
$x_1,\ x_2,\ldots, x_\ell, x_{-\ell}, x_{-\ell+1},\ldots, x_{-1}$ with coefficients in $k$ (i.e., $\,{\mathcal P}_{2\ell}=k[E(\varpi)]$),
then
\begin{equation}\label{examp1}
{\mathcal A}(\varpi)={\mathcal P}_{2\ell}/(f).
\end{equation}
The $k$-algebra ${\mathcal A}(\varpi)$ is a unique factorization domain of transcendence degree $n:=2\ell-1$ over $k$,
and ${\mathcal A}(\varpi)^*=k^*$.\;The hypersurface  of zeros of $f$ in
 $V$ is not smooth, hence  ${\mathcal A}(\varpi)$ is not a polynomial ring over $k$.

 Identifying
every element of ${\rm GL}(V)$ with its matrix in basis \eqref{basis}, we may assume that
${\rm GL}(V)={\rm GL}_{2\ell}$
and that the elements of $\varphi_{\varpi}(T)$ (resp.  $\varphi_{\varpi}(B)$)
are diagonal (resp.\;upper triangular) matrices (see, e.g., \cite[Chap.\,VIII, \S{13}, no.\,4]{B75}).\;Using the explicit description of $\Phi$, $\Delta$, and $U_\alpha$'s available in this case (see loc.cit.),
 it is then not difficult to see that all the derivations $D_{-\alpha}$ of $\mathcal P_{2\ell}$, where $\alpha\in \Phi^+\setminus \Phi_{I(\varpi)}$,
 are precisely
 the following
 $n-1$ commuting derivations
 $D_j$, $j=2, 3,\ldots, \ell, -\ell,\ldots, -3, -2$, defined by the formula
\begin{align*}
D_j(x_i)&=\begin{cases}0 &\mbox{for $i\neq j$,}\\
x_1&\mbox{for $i=j$}\end{cases}\quad \mbox{if $i\neq -1$},\\
D_j(x_{-1})&=-x_{-j}.
\end{align*}

Let $\partial_j$ be the locally nilpotent derivation of  $\mathcal A(\varpi)$ induced (in view of
$D_j(f)=0$ and \eqref{examp1}) by $D_j$.\;Then $D:=\{\partial_j\}$ is the set of $n-1$ commuting derivations
that are linearly independent over $\mathcal A(\varpi)$; whence
\eqref{cond} holds (see, e.g., \cite[Prop.\;3.4]{Ma03}, \cite[Lemma 1]{DEFM11}).\;Thus
in this case the answer to
Question\;\boldmath$1$ is negative.\quad $\square$

\end{example}
\begin{example}\label{ex2}
Let $G$ be of type ${\sf B}_{\ell}$, $\ell\geqslant 2$, and $\varpi=\varpi_1$.\;The argument similar to that in Example \ref{ex1} shows
that if $\mathcal P_{2\ell+1}$ is the polynomial ring in $2\ell+1$ variables   $x_1,\ x_2,\ldots, x_\ell, x_0, x_{-\ell}, x_{-\ell+1},\ldots, x_{-1}$ with coefficients in $k$,  then
\begin{equation}\label{h}
\mathcal A(\varpi)=\mathcal P_{2\ell+1}/(h),\quad\mbox{where $h=x_0^2+x_{-1}x_1+x_{-2}x_2+\cdots + x_{-\ell}x_\ell$.}
\end{equation}
The $k$-algebra $\mathcal A(\varpi)$ is a unique factorization domain of transcendence degree $n:=2\ell$ over $k$, which is not a polynomial ring over $k$, and  $\mathcal A(\varpi)^*=k^*$.\;All the derivations $D_{-\alpha}$ of $\mathcal P_{2\ell+1}$, where $\alpha\in \Phi^+\setminus \Phi_{I(\varpi)}$,
 are precisely
 the following
 $n-1$ commuting derivations
 $D_j$, $j=2, 3,\ldots, \ell, 0, -\ell,\ldots, -3, -2$, defined by the formula
\begin{align*}
D_j(x_i)&=\begin{cases}0 &\mbox{for $i\neq j$,}\\
x_1&\mbox{for $i=j$}\end{cases}\quad \mbox{if $i\neq -1$},\\
D_j(x_{-1})&=\begin{cases}-x_{-j}&\mbox{for $j\neq 0$,}\\
2x_0&\mbox{for $j=0$}.
\end{cases}
\end{align*}

Let $\partial_j$ be the locally nilpotent derivation of  $\mathcal A(\varpi)$ induced (in view of
$D_j(h)=0$ and \eqref{h}) by $D_j$.\;Then $D:=\{\partial_j\}$ is the set of $n-1$ commuting derivations
that are linearly independent over $\mathcal A(\varpi)$; whence
\eqref{cond} holds.
Therefore, in this case the answer to
Question\;\boldmath$1$ is negative as well.\quad $\square$
\end{example}

  In Examples \ref{ex1} and \ref{ex2}, the algebras $\mathcal A(\varpi)$  are hypersurfaces (quadratic cones).\;In the general case,
  they are factor algebras
  of polynomial algebras modulo the ideals generated by finitely many quadratic forms.\;Namely, the $G$-module ${\rm S}^2(E(\varpi)^*)$ of  quadratic forms on $E(\varpi)$ contains a unique submodule
 (the Cartan component) $C(\varpi)$ isomorphic to $E(2\varpi)^*$; whence there is a unique submodule $M(\varpi)$ such that ${\rm S}^2(E(\varpi)^*)=C(\varpi)\oplus M(\varpi)$.\;It is known that the
 ideal of $k[E(\varpi)]$
 generated by $M(\varpi)$ is then the ideal of elements $k[E(\varpi)]$ vanishing on
  of $X(\varpi)$.\;Therefore,
$X(\varpi)$ is cut out in $E(\varpi)$\;by
 $$\frac{\dim E(\varpi)\big(\dim E(\varpi)+1\big)}{2}-\dim E(2\varpi)$$
 homogeneous quadrics (cf.\;\cite{Li82}).

 We note that a pair $(\mathcal A, D)$ with $\mathcal A$ of transcendence degree $3$ over $k$, for which the answer to Question 1 is negative, exists as well: basing on the famous theorem that the Koras--Russell threefold $X$ is not isomorphic to ${\bf A}\!^3$ (see  \cite{M.-L.96}), in \cite{EK05} is shown that one may take $\mathcal A=k[X]$.

 \section{Remarks}\label{rema}

1. The same arguments as in the proof of Theorem \ref{raco} prove the following

\begin{theorem}
Let $X$ be an irreducible affine $n$-dimensional variety endowed with a regular action of
a unipotent algebraic group $U$.\;Assume that
\begin{enumerate}[\hskip 4.2mm\rm(i)]
\item $X$ is unirational;
\item $X$ is normal;
\item $k[X]^\star=k^\star$;
\item $\underset{x\in X}{\max} \dim U\cdot x=n-1.$
\end{enumerate}
 Then there is an irreducible element $t$ of $k[X]$ and the elements
 $f_1,\ldots f_{n-1}\in k(X)$ such that
\begin{enumerate}[\hskip 4.2mm\rm(a)]
 \item $k[X]^U=k[f]$;
 \item $k(X)=k(t, f_1,\ldots, f_{n-1})$.
 \end{enumerate}
 In particular, $X$ is rational.
\end{theorem}

2. Theorem 1 in \cite{DEFM11} reads as follows:

\vskip 2mm

\noindent {\it Let $U$ be an $n$-dimensional unipotent group acting faithfully on an affine $n$-dimensional variety $X$ satisfying $\mathcal O(X)^\star=k^\star$.\;Then $X\cong \bAn$ if one of the following two conditions hold}:
\begin{enumerate}[\hskip 4.2mm \rm(a)]
\item {\it some $x\in X$ has trivial isotropy subgroup, or}
\item {$n=2$, $X$ is factorial, and $U$ acts without fixed points}.
\end{enumerate}

\vskip 2mm

The proof shows that, in fact,
$X$ is also assumed to be irreducible.\;We remark that,
actually, given (a),  the assumption $\mathcal O(X)^\star=k^\star$ is superfluous and, changing the proof, one may drop it.\;Moreover, in this case, more generally, affiness of $X$ may be replaced by quasi-affiness, the assumption
$\dim U=n$ may be dropped,
and (a) may be replaced by the assumption
\begin{equation}\label{hp}
\dim U_x+\dim X=\dim U.
\end{equation}

Indeed, \eqref{hp} implies that $\dim U\cdot x=\dim X$.\;On the other hand, by \cite[Thm.\,2]{Ro61_2}, unipotency of $U$ implies that
$U\cdot x$ is closed in $X$.\;Hence $U\cdot x=X$.\;Therefore, $X\cong U/U_x$, whence the claim by (i) in Introduction.

 \end{document}